\documentclass[12pt]{article}
\usepackage{amssymb}
\usepackage[all,arc]{xy}
\usepackage{mathrsfs}
\usepackage{amsfonts}
\usepackage{extarrows}
\usepackage{amsmath,amsthm,amsfonts,amssymb,amscd,color}

\allowdisplaybreaks[4]
\newtheorem{defn}{Definition}[section]
\newtheorem{them}[defn]{Theorem}
\newtheorem{lem}[defn]{Lemma}

\newtheorem{prop}[defn]{Proposition}

\newtheorem{exam}[defn]{Example}
\newtheorem{con}[defn]{Conjecture}
\newtheorem{prob}[defn]{Question}

\newtheorem{rem}[defn]{Remark}

\numberwithin{equation}{section}

\begin{document}
\title{Some results of strongly primitive tensors\footnote{L. You's research is supported by the National Natural Science Foundation of China
(Grant No.11571123) and the Guangdong Provincial Natural Science Foundation (Grant No.2015A030313377),
P. Yuan's research is supported by the NSF of China (Grant No. 11671153).}}

\author{ Lihua You\footnote{{\it{Corresponding author:\;}}ylhua@scnu.edu.cn.} 
 \qquad Yafei Chen\footnote{{\it{Email address:\;}}764250280@qq.com. } 
 \qquad Pingzhi Yuan\footnote{{\it{Corresponding author:\;}}yuanpz@scnu.edu.cn. }
 }\vskip.2cm
\date{{\small
School of Mathematical Sciences, South China Normal University,\\
Guangzhou, 510631, P.R. China\\
}}
\maketitle
\noindent {\bf Abstract } In this paper, we show that an order $m$ dimension 2 tensor is primitive if and only if its majorization matrix is primitive,  and then we obtain   the characterization of order $m$ dimension 2 strongly primitive tensors and   the bound of the strongly primitive degree. Furthermore, we study the properties of strongly primitive tensors with $n\geq 3$, and propose some problems for further research.



\section{Introduction}

\hskip.6cm A  nonnegative square matrix $A= (a_{ij})$ of order $n$ is nonnegative primitive (or simply, primitive) if $A^k>0$ for
some positive integer $k$. The least such $k$ is called the primitive exponent (or simply, exponent) of $A$ and is
denoted by $\exp(A)$.

Since the work of Qi \cite{2005Q} and Lim \cite{2005L}, the study  of tensors which regarded as the generalization of matrices,
the spectra of tensors (and hypergraphs) and their various applications has attracted much attention and interest.

As is in  \cite{2005Q}, an order $m$ dimension $n$ tensor $\mathbb{A}= (a_{i_1i_2\ldots i_m})_{1\le i_j\le n \hskip.2cm (j=1, \ldots, m)}$
over the complex field $\mathbb{C}$ is a multidimensional array with all entries
$$a_{i_1i_2\ldots i_m}\in\mathbb{C}\, ( i_1, \ldots, i_m\in [n]=\{1, \ldots, n\}).$$
In \cite{2008C} and \cite{2011C}, Chang et al investigated the properties of the spectra of nonnegative tensors,
defined the irreducibility of tensors and the primitivity of nonnegative tensors (as Definition \ref{defn11}),
and extended many important properties of primitive matrices to primitive tensors.

\begin{defn} {\rm (See \cite{2011C})} \label{defn11} Let $\mathbb{A}$ be a nonnegative  tensor with order $m$ and dimension $n$,
$x=(x_1, x_2, \ldots, x_n)^T\in\mathbb{R}^n$ a vector and $x^{[r]}=(x_1^r, x_2^r, \ldots, x_n^r)^T$.
Define the map $T_\mathbb{A}$ from $\mathbb{R}^n$ to $\mathbb{R}^n$ as: $T_\mathbb{A}(x)=(\mathbb{A}x)^{[\frac{1}{m-1}]}$.
If there exists some positive integer $r$ such that $T_\mathbb{A}^r(x)>0$ for all nonnegative nonzero vectors $x\in\mathbb{R}^n$,
then $\mathbb{A}$ is called primitive and the smallest such integer $r$ is called the primitive degree of $\mathbb{A}$,
denoted by $\gamma(\mathbb{A})$. \end{defn}

Recently, Shao \cite{2013S} defined the general product of two n-dimensional tensors as follows.
\begin{defn}{\rm (See \cite{2013S})} \label{defn12}
Let $\mathbb{A}$ {\rm (}and $\mathbb{B}${\rm)} be an order $m\ge2$ {\rm (}and $k\ge 1${\rm)}, dimension $n$ tensor, respectively.
Define the general product  $\mathbb{A}\mathbb{B}$ to be the following tensor $\mathbb{D}$ of order $(m-1)(k-1)+1$ and dimension $n$:
$$ d_{i\alpha_1\ldots\alpha_{m-1}}=\sum\limits_{i_2, \ldots, i_m=1}^na_{ii_2\ldots i_m}b_{i_2\alpha_1}\ldots b_{i_m\alpha_{m-1}}
\quad (i\in[n], \, \alpha_1, \ldots, \alpha_{m-1}\in[n]^{k-1}).$$
\end{defn}

The tensor product  is a generalization of the usual matrix product, and satisfies  a very useful property:
the associative law (\cite{2013S}, Theorem 1.1).
With the general product,  when $k=1$ and $\mathbb{B}=x=(x_1,\ldots, x_n)^T\in \mathbb{C}^n$ is a vector of dimension $n$,
then $\mathbb{A}\mathbb{B} = \mathbb{A}x$ is still a vector of dimension $n$, and for any $i\in [n],$
$(\mathbb{A}\mathbb{B})_i=(\mathbb{A}x)_i=\sum\limits_{i_2, \ldots, i_m=1}^na_{ii_2\ldots i_m}x_{i_2}\ldots x_{i_m}.$

As an application of the general tensor product defined by Shao \cite{2013S},
Shao presented a simple characterization of the primitive tensors.
Now we  give the definition of ``essentially positive" which introduced by Pearson.

\begin{defn} {\rm (See \cite{2010P}, Definition 3.1)} \label{defn13}A nonnegative tensor $\mathbb{A}$ is  called essentially positive,
if for any nonnegative nonzero vector $x\in \mathbb{R}^n, \mathbb{A}x>0$ holds.
\end{defn}

\begin{prop}{\rm (See \cite{2013S}, Proposition 4.1)} \label{pro14}
Let $\mathbb{A}$ be an order $m$ and dimension $n$ nonnegative  tensor. Then the following three conditions are equivalent:

{\rm (1) } For any $i, j\in[n], a_{ij\ldots j}>0$ holds.

{\rm (2) } For any $j\in[n],  \mathbb{A}e_j>0$ holds {\rm(}where $e_j$ is the $j$-{th} column of the identity matrix $I_n${\rm)}.

{\rm (3) } For any nonnegative nonzero vector $x\in \mathbb{R}^n, \mathbb{A}x>0$ holds. \end{prop}

By Proposition \ref{pro14}, the following Definition \ref{defn15} is equivalent to Definition \ref{defn13}.
What's more, in Proposition \ref{pro16}, Shao showed a characterization of primitive tensors and defined the primitive degree by using
the properties of tensor product and the zero patterns which defined by Shao in \cite{2013S}.

\begin{defn} {\rm (See \cite{2013S}, Definition 4.1)} \label{defn15}
A nonnegative tensor $\mathbb{A}$ is  called essentially positive, if it satisfies one of the three conditions in Proposition \ref{pro14}.\end{defn}

\begin{prop}\label{pro16}{\rm (See \cite{2013S}, Theorem 4.1)}
A nonnegative tensor $\mathbb{A}$ is primitive if and only if there exists some positive integer $r$ such that $\mathbb{A}^r$ is essentially positive. Furthermore, the smallest such $r$ is the primitive degree of $\mathbb{A}$, $\gamma(\mathbb{A})$.
\end{prop}

The concept of the majorization matrix of a tensor introduced by Pearson is very useful.

\begin{defn}\label{defn17}{ \rm (See \cite{2010P},  Definition 2.1)}
 The majorization matrix  $M(\mathbb{A})$ of the tensor $\mathbb{A}$ is defined as  $(M(\mathbb{A}))_{ij}=
a_{ij\ldots j}, i, j\in[n]$.
\end{defn}

By  Definition \ref{defn15}, Proposition \ref{pro16} and  Definition \ref{defn17},
 the following characterization of the primitive tensors was easily obtained.

\begin{prop}\label{pro18}{\rm (See \cite{2014Y}, Remark 2.6)}
Let $\mathbb{A}$ be a nonnegative  tensor with order $m$ and dimension $n$.
Then $\mathbb{A}$ is primitive if and only if there exists some positive integer $r$ such that $M(\mathbb{A}^r)>0.$
Furthermore, the smallest such $r$ is the primitive degree of $\mathbb{A}$, $\gamma(\mathbb{A})$.
\end{prop}

On the primitive degree $\gamma(\mathbb{A})$, Shao proposed  the following conjecture for further research.%

\begin{con} \label{con19}{\rm (See \cite{2013S},  Conjecture 1)}
When $m$ is fixed, then there exists some polynomial $f(n)$ on $n$ such that $\gamma(\mathbb{A})\le f(n)$ for
all nonnegative primitive tensors of order $m$ and dimension $n$. \end{con}

In the case of $m = 2$ ($\mathbb{A}$ is a matrix), the well-known Wielandt's upper bound tells us that we can take
$f(n) =(n-1)^2 + 1.$
Recently,  the authors \cite{2014Y} confirmed Conjecture \ref{con19} by proving Theorem \ref{thm110}.

\begin{them}\label{thm110}{\rm (See \cite{2014Y}, Theorem 1.2)}
Let $\mathbb{A}$ be a nonnegative primitive tensor with order $m$ and dimension $n$.
Then its primitive degree $\gamma(\mathbb{A})\le (n-1)^2+1$, and the upper bound is tight.
\end{them}

 They also showed that there are no gaps  in the tensor case   in \cite{2015Y},
 which implies that the result of the case $m\geq 3$ is totally different from the case $m=2$.
In \cite{2013S}, Shao also proposed the concept of strongly primitive  for further research.

\begin{defn}\label{defn111}{\rm (See \cite{2013S}, Definition 4.3)}
Let $\mathbb{A}$ be a nonnegative tensor with order $m$ and dimension $n$. If there exists some positive
integer $k$ such that $\mathbb{A}^k > 0$ is a positive tensor, then $\mathbb{A}$ is called strongly primitive,
and the smallest such $k$ is called the strongly primitive degree of $\mathbb{A}$.
\end{defn}

Let $\mathbb{A}=(a_{i_1i_2\ldots i_m})$ be a nonnegative tensor with order $m$ and dimension $n$.
It is clear that if $\mathbb{A}$ is strongly primitive, then $\mathbb{A}$ is primitive.
For convenience,  let $\eta(\mathbb{A})$ be the strongly primitive degree of $\mathbb{A}$.
Clearly, $\gamma(\mathbb{A})\leq \eta(\mathbb{A})$.
In fact, it is obvious that in the matrix case $(m = 2)$,
 a nonnegative matrix $A$ is primitive if and only if $A$ is strongly primitive, and $\gamma(A)=\eta(A)=\exp(A)$.
But in the case $m\geq 3$, Shao gave an example to show  that these two concepts are not equivalent in the case $m\geq 3$.
In \cite{2015Y}, the authors proposed the following question.

\begin{prob}\label{prob112}{\rm(\cite{2015Y}, Question 4.18)}
Can we define and  study the strongly primitive degree, the strongly primitive degree set,
the $j$-strongly primitive of strongly primitive tensors and so on?
\end{prob}

Based on Question \ref{prob112}, we  study  primitive tensors and  strongly primitive tensors in this paper,
show that an order $m$ dimension 2 tensor is primitive if and only if its majorization matrix is primitive,
and  obtain   the characterization of order $m$ dimension 2 strongly primitive tensors and   the bound of the strongly primitive degree. Furthermore, we study the properties of strongly primitive tensors with $n\geq 3$, and propose some problems for further research.

\section{Preliminaries}

\hskip.6cm In \cite{2015Y}, the authors obtained the following Proposition \ref{prop21}
and gave  Example \ref{exam23} by computing the strongly primitive degree.

\begin{prop}\label{prop21}{\rm(\cite{2015Y}, Proposition 4.16)}
Let $\mathbb{A}=(a_{i_1i_2\ldots i_m})$ be a nonnegative strongly primitive tensor with order $m$ and dimension $n$.
Then for any $\alpha\in [n]^{m-1}$, there exists some $i\in [n]$ such that $a_{i\alpha}>0$.
\end{prop}

Let $k(\geq 0)$, $n(\geq 2)$, $q(\geq 0)$, $r(\geq 1)$ be integers and $k=(n-1)q+r$ with $1\leq r\leq n-1$  when $k\geq 1$.
In \cite{2015H, 2014Y, 2015Y}, the authors defined some nonnegative tensors with order $m$ and dimension $n$ as follows:
 $$
 \mathbb{A}_k=(a^{[k]}_{i_1i_2\ldots i_m})_{1\leq i_j\leq n \hskip.2cm (j=1,\ldots, m)},$$
\noindent where

{\rm(1) } $M(\mathbb{A}_k)
=M_1=\left(
                 \begin{array}{cccccc}
                    0 & 0 & \ldots & 0 & 1 & 1 \\
                    1 & 0 & \ldots & 0 & 0 & 0 \\
                    0 & 1 & \ldots & 0 & 0 & 0\\
                    \vdots &\vdots & \ddots & \vdots & \vdots &\vdots \\
                    0 & 0 & \ldots & 1 & 0 & 0 \\
                    0 & 0 & \ldots & 0 & 1 & 0 \\
                 \end{array}
               \right).$\vskip.1cm

{\rm (2) } $a^{[0]}_{ii_2\ldots i_m}=0$, if $i_2\ldots i_m\not=i_2\ldots i_2$ for any $i\in[n]$.\vskip.1cm

{\rm (3) }$a^{[k]}_{i\alpha}=1$, if $i\in[n]\backslash\{r-q, r-q+1, \ldots, r,r+1\} \pmod n$ and $\alpha=i_2\ldots i_m\in[n]^{m-1}$ with $\{i_2,\ldots,i_m\}=\{r-q-1,r\} \pmod n;$

{\rm (4) }$a^{[k]}_{i_1i_2\ldots i_m}=0$, except for {\rm (1)} and {\rm (3)}.
\vskip0.2cm
The authors showed the tensors $\mathbb{A}_k$ $(k\geq 0)$ are primitive,  the  primitive degree
$\gamma(\mathbb{A}_0)=(n-1)^2+1$(\cite{2014Y}) and  $\gamma(\mathbb{A}_k)=k+n$ (\cite{2015Y}, Theorem 3.3) for $1\leq k\leq n^2-3n+2$.

\begin{rem}\label{rem22}
It is clear that for any $\mathbb{A}_k$ $(0\leq k\leq n^2-3n+2)$, there exists some $\alpha\in [n]^{m-1}$,
for any $i\in [n]$,  $a^{[k]}_{i\alpha}=0$. Thus for each $0\leq k\leq n^2-3n+2$,
$\mathbb{A}_k$  is not a strongly primitive tensor by Proposition \ref{prop21}.
\end{rem}

\begin{exam}\label{exam23}{\rm(\cite{2015Y}, Example 4.17)}
Let $m=n=3$, $\mathbb{A}=(a_{i_1i_2i_3})$ be a nonnegative tensor with order $m$ and dimension $n$,
where $a_{111}=a_{222}=a_{333}=a_{233}=a_{311}=0$ and other $a_{i_1i_2i_3}=1$. Then $\eta(\mathbb{A})=4.$
\end{exam}

\begin{rem}\label{rem24}
In fact, we can obtain $\gamma(\mathbb{A})=\eta(\mathbb{A})=4$ because of $a_{233333333}^{(3)}=0$,
where $\mathbb{A}^3=(a_{i_1i_2\ldots i_9}^{(3)})$.
\end{rem}

In the computation of Example \ref{exam23}, we note that the following equation is useful and will be used repeatedly.
It can be easy to obtain by  the general product of two n-dimensional tensors which defined in Definition \ref{defn12} in \cite{2013S}.

Let $\mathbb{A}$ be a nonnegative primitive tensor with order $m$ and dimension $n$,
$\alpha_{2}, \ldots, \alpha_{m}\in [n]^{(m-1)^{k-1}}.$  Then we have
  \begin{equation}\label{eq21}
(\mathbb{A}^{k})_{i\alpha_{2}\ldots\alpha_{m}}
=\sum\limits_{i_{2},i_{3},\ldots, i_{m}=1}^{n} a_{ii_{2}i_{3}\ldots i_{m}}(\mathbb{A}^{k-1})_{i_{2}\alpha_{2}}\ldots
(\mathbb{A}^{k-1})_{i_{m}\alpha_{m}}.
\end{equation}

\begin{prop}\label{prop25}{\rm (See \cite {2014Y}, Proposition 2.7)}
 Let $\mathbb{A}$ be a nonnegative primitive tensor with order $m$ and dimension $n$,
 $M(\mathbb{A})$ be the majorization matrix of $\mathbb{A}$. Then we have:

 {\rm (1) } For each $j\in [n]$, there exists an integer $i\in [n]\backslash \{j\}$ such that $(M(\mathbb{A}))_{ij}>0$.

 {\rm (2) } There exist some $j\in [n]$ and integers $u,v$ with $1\leq u<v\leq n$ such that $(M(\mathbb{A}))_{uj}>0$ and $(M(\mathbb{A}))_{vj}>0$.
\end{prop}

Let  $\alpha=jj\ldots j \in [n]^{m-1}$, then  $M(\mathbb{A})_{ij}=a_{i\alpha}$.
We can see that Proposition \ref{prop21} is the generalization of the result (1) of Proposition \ref{prop25}
from a primitive tensor to a strongly primitive tensor.
We note that  Proposition \ref{prop25} played an important role in \cite {2014Y},
and if  $\mathbb{A}$ is a nonnegative strongly primitive tensor, then  $\mathbb{A}$ must be a nonnegative primitive tensor,
thus the result (2) of Proposition \ref{prop25} also holds for nonnegative strongly primitive tensors.

\begin{prop}\label{prop26} Let $\mathbb{A}=(a_{i_{1}i_{2}\ldots i_{m}})$ be a nonnegative strongly primitive tensor with order $m$ and
dimension $n$. Then there exists at least one $j\in [n]$ and integers $u,v$ with $1\leq u<v\leq n$ such that $(M(\mathbb{A}))_{uj}>0$ and $(M(\mathbb{A}))_{vj}>0$.
\end{prop}

\begin{prop}\label{prop27}
Let $\mathbb{A}=(a_{i_{1}i_{2}\ldots i_{m}})_{1\leq i_{j}\leq n(j=1,\ldots, m)}$ be a nonnegative tensor with order $m$ and dimension $n$ and
$\mathbb{A}\not=\mathbb{J}$.
For given $i\in [n]$, if  $a_{i\alpha}=a_{jii\ldots i}=1$ for any  $\alpha \in[n]^{m-1}$ and any $j\in[n]\backslash \{i\}$,
then $\mathbb{A}$ is strongly primitive with $\eta(\mathbb{A})=2$.
\end{prop}
\begin{proof}
By (\ref{eq21}), for any $k\in [n]$ and $\alpha_2, \ldots, \alpha_m \in[n]^{m-1}$, we have

$(\mathbb{A}^{2})_{k\alpha_{2}\ldots\alpha_{m}}=\sum\limits_{k_{2},k_{3},\ldots, k_{m}=1}^{n} a_{kk_{2}\ldots k_{m}}a_{k_{2}\alpha_{2}}\ldots a_{k_{m}\alpha_{m}}\geq a_{kii\ldots i}a_{i\alpha_{2}}\ldots a_{i\alpha_{m}}=1$,
which implies $\mathbb{A}$ is strongly primitive and $\eta(\mathbb{A})=2$.
\end{proof}

\begin{rem}\label{rem28}
From Proposition \ref{prop27}, we can see that:

{\rm(1)} There exist at least  $n(2^{(n-1)(n^{m-1}-1)}-1)$    strongly primitive tensors such that its strongly primitive degree is equal to 2.

{\rm(2)} We cannot improve the result  of Proposition \ref{prop21}
any more by the fact that  there exists $i\in [n]$ such that  $a_{i\alpha}=1>0$ for any  $\alpha \in[n]^{m-1}$ and
there is exactly one $i$ such that  $a_{i\alpha}>0$ for any  $\alpha\not=ii\ldots i$.

{\rm(3)} Similarly, we cannot improve the result  of Proposition \ref{prop26} any more
by the fact that  there is exactly one $i\in [n]$  such that $(M(\mathbb{A}))_{ui}>0$ for any $u\in [n]$
and for any other $j\in[n]\backslash \{i\}$, there exists only $i\in[n]$ such that $(M(\mathbb{A}))_{ij}>0$.

{\rm(4)} What's more, combining the above arguments, we know whether a nonnegative tensor is a nonnegative strongly primitive tensor or not,
and the value of the strongly primitive degree of a nonnegative strongly primitive tensor  do not depend on the number of nonzero entries,
but the positions of the nonzero entries.
\end{rem}

\begin{prop}\label{prop29}
Let $\mathbb{A}=(a_{i_{1}i_{2}\ldots i_{m}})$ be a nonnegative strongly primitive tensor with order $m$ and dimension $n$.
Then for any $i\in [n]$, there exists some $\alpha \in [n]^{m-1}$ such that $a_{i\alpha}>0$.
\end{prop}
\begin{proof}
Since $\mathbb{A}$ is strongly primitive,
there exists some $k>0$ such that $\mathbb{A}^{k}>0$ by Definition \ref{defn11}.
Assume that there exists some $i\in [n]$ such that $a_{i\alpha}=0$ for any $\alpha \in [n]^{m-1}$. Then by (\ref{eq21}),  we have
\begin{equation*}
(\mathbb{A}^{k})_{i\alpha_{2}\ldots\alpha_{m}}=\sum\limits_{i_{2},i_{3},\ldots,i_{m}=1}^{n} a_{ii_{2}i_{3}\ldots i_{m}}(\mathbb{A}^{k-1})_{i_{2}\alpha_{2}}\ldots (\mathbb{A}^{k-1})_{i_{m}\alpha_{m}}=0,
\end{equation*}
which leads to a contraction.
\end{proof}

\begin{rem}\label{rem210}
Let $\mathbb{A}=(a_{i_{1}i_{2}\ldots i_{m}})_{1\leq i_{j}\leq n(j=1,\ldots, m)}$ be a nonnegative tensor with order $m$ and dimension $n$.
For given $i\in [n]$, we take  $a_{i\alpha}=a_{jii\ldots i}=1$ for any  $\alpha \in[n]^{m-1}$ and any $j\in [n]\setminus \{i\}$,
and any other entry $a_{i_{1}i_{2}\ldots i_{m}}=0$. Then $\mathbb{A}$ is strongly primitive with $\eta(\mathbb{A})=2$ by Proposition \ref{prop27}.
This implies that we cannot improve the result  of Proposition \ref{prop29} any more,
and it indicates the importance of the positions of the nonzero entries again.
\end{rem}

\begin{prop}\label{prop211} Let $\mathbb{A}$ be a nonnegative strongly primitive tensor and $k=\eta(\mathbb{A})$.
Then for any integer $t>k>0$, we have $\mathbb{A}^t>0$.
\end{prop}
\begin{proof}
It is clear that  $\mathbb{A}^k>0$ by $k=\eta(\mathbb{A})$. We only need to show  $\mathbb{A}^{k+1}>0$,
say, for any $i\in [n]$,  and
any $\alpha_{2}, \ldots, \alpha_{m}\in [n]^{{(m-1)}^{k}}$, we show   $(\mathbb{A}^{k+1})_{i\alpha_{2}\ldots\alpha_{m}}>0$.

By Proposition \ref{prop29}, there  exists some $\alpha=j_2j_3\ldots j_m \in [n]^{m-1}$ such that $a_{i\alpha}=a_{ij_2\ldots j_m}>0$.
By  $\mathbb{A}^k>0$ we have $(\mathbb{A}^{k})_{j_2\alpha_{2}}>0$, $\ldots,$ $(\mathbb{A}^{k})_{j_m\alpha_{m}}>0$, then by (\ref{eq21}), we have

$(\mathbb{A}^{k+1})_{i\alpha_{2}\ldots\alpha_{m}}=\sum\limits_{i_{2},i_{3},\ldots,
i_{m}=1}^{n} a_{ii_{2}i_{3}\ldots i_{m}}(\mathbb{A}^{k})_{i_{2}\alpha_{2}}(\mathbb{A}^{k})_{i_{3}\alpha_{3}}\ldots (\mathbb{A}^{k})_{i_{m}\alpha_{m}}$

\hskip2.5cm $\geq a_{ij_{2}j_{3}\ldots j_{m}}(\mathbb{A}^{k})_{j_{2}\alpha_{2}}\ldots (\mathbb{A}^{k})_{j_{m}\alpha_{m}}$

\hskip2.5cm $>0.$
\end{proof}

\begin{prop}\label{prop212} Let $\mathbb{A}$ be a nonnegative tensor with order $m$ and dimension $n$,
and  $t$ be a positive integer. Then $\mathbb{A}$ is strongly primitive if and only if $\mathbb{A}^{t}$ is strongly primitive.
\end{prop}
\begin{proof}
Firstly, the sufficiency is obvious. Now we show the necessity.
Let $k=\eta(\mathbb{A})$. Then $\mathbb{A}^k>0$ by   $\mathbb{A}$ is strongly primitive.
Let  $s$ be a positive integer such that $st\geq k$, then   $\mathbb{A}^{st}>0$ by Proposition \ref{prop211}.
Thus  $(\mathbb{A}^{t})^s=\mathbb{A}^{st}>0$, which implies $\mathbb{A}^{t}$ is strongly primitive.
\end{proof}

\section{A characterization   of the (strongly) primitive tensor with order $m$ and dimension $2$}
\hskip.6cm In this section, we  study  primitive tensors and  strongly primitive tensors in this paper,
show that an order $m$ dimension 2 tensor is primitive if and only if its majorization matrix is primitive,
and  obtain   the characterization of order $m$ dimension 2 strongly primitive tensors and   the bound of the strongly primitive degree.

\begin{lem}\label{lem31}{\rm(See \cite {2013S}, Corollary 4.1)}
 Let $\mathbb{A}$ be a nonnegative tensor with order $m$ and dimension $n$.
 If $M(\mathbb{A})$ is primitive, then $\mathbb{A}$ is also primitive and  in this case,
 we have $\gamma(\mathbb{A})\leq \gamma(M(\mathbb{A}))\leq (n-1)^{2}+1$.
\end{lem}
\begin{them}\label{thm32}
Let $\mathbb{A}$ be a nonnegative tensor with order $m$ and dimension $n=2$. Then $\mathbb{A}$ is primitive if and only if $M(\mathbb{A})$ is primitive.
\end{them}
\begin{proof}
Firstly, the sufficiency is obvious by Lemma \ref{lem31}.
Now we only show the necessity. Clearly,   all primitive matrices  of order 2 are listed as follows:
\begin{equation}\label{eq31}
\left(\begin{array}{cc}
1 & 1\\
1 & 0
\end{array}\right),
\left(\begin{array}{cc}
0 & 1\\
1 & 1
\end{array}\right),
\left(\begin{array}{cc}
1 & 1\\
1 & 1
\end{array}\right).
\end{equation}

Let $\mathbb{A}$ be primitive.  Then $\gamma(\mathbb{A})\leq 2$ by Theorem \ref{thm110} and $M(\mathbb{A}^{2})>0$ by Proposition \ref{pro18}.
Now we assume that $M(\mathbb{A})$ is not primitive, we will show $\mathbb{A}$ is also not primitive.

It is not difficult to find that

$(\mathbb{A}^{2})_{ijj\ldots j}$

$=\sum\limits_{i_{2},i_{3},\ldots, i_{m}=1}^{2} a_{ii_{2}i_{3}\ldots i_{m}}a_{i_{2}jj\ldots j}\ldots a_{i_{m}jj\ldots j}$ 
\begin{eqnarray}
\noindent=a_{i22\ldots 2}(a_{2jj\ldots j})^{m-1}+
    \sum\limits_{i_{2},i_{3},\ldots, i_{m}=1, i_2i_3\ldots i_m\not=22\ldots 2}^{2}
    a_{ii_{2}\ldots i_{m}}a_{i_{2}jj\ldots j}\ldots a_{i_{m}jj\ldots j} \label{eq32} \end{eqnarray}
\begin{eqnarray}
=a_{i11\ldots 1}(a_{1jj\ldots j})^{m-1}+\sum\limits_{i_{2},i_{3},\ldots, i_{m}=1, i_2i_3\ldots i_m\not=11\ldots 1}^{2}
    a_{ii_{2}\ldots i_{m}}a_{i_{2}jj\ldots j}\ldots a_{i_{m}jj\ldots j} \label{eq33}.
\end{eqnarray}

In (\ref{eq32}), we note that $i_2i_3\ldots i_m\not=22\ldots 2$, which implies that there exists at least one entry, say,
$i_s=1$ where $2\leq s\leq m$,
then $ a_{1jj\ldots j}\in \{a_{i_{2}jj\ldots j}, \ldots, a_{i_{m}jj\ldots j}\}.$

Similarly, in (\ref{eq33}), we note that $i_2i_3\ldots i_m\not=11\ldots 1$, which implies that there exists at least one entry, say,
$i_s=2$ where $2\leq s\leq m$,
then $ a_{2jj\ldots j}\in \{a_{i_{2}jj\ldots j}, \ldots, a_{i_{m}jj\ldots j}\}.$

Thus, by (\ref{eq32}), (\ref{eq33}) and the above arguments, we have
\begin{eqnarray}
M(\mathbb{A}^{2})_{ij}=(\mathbb{A}^{2})_{ijj\ldots j}&=& a_{i22\ldots 2}(a_{2jj\ldots j})^{m-1}+a_{1jj\ldots j}P \label{eq34}\\
&=& a_{i11\ldots 1}(a_{1jj\ldots j})^{m-1}+a_{2jj\ldots j}Q  \label{eq35}.
\end{eqnarray}

Since $M(\mathbb{A})$ is not primitive, by (\ref{eq31}), we can complete the proof by the following two cases.

\noindent {\bf Case 1: }
$M(\mathbb{A})=\left(\begin{array}{cc}
* & 0\\
* & *
\end{array}\right)
or \left(\begin{array}{cc}
* & *\\
0 & *
\end{array}\right)$.

\noindent \textbf{Subcase 1.1: }$M(\mathbb{A})_{12}=0$.

Then $a_{12\ldots 2}=0$. By (\ref{eq34}), we have $M(\mathbb{A}^{2})_{12}=(\mathbb{A}^{2})_{12\ldots 2}=0$, which implies $\mathbb{A}^{2}$ is not essential positive.

\noindent \textbf{Subcase 1.2: }$M(\mathbb{A})_{21}=0$.

Then $a_{21\ldots 1}=0$. By (\ref{eq35}), we have $M(\mathbb{A}^{2})_{21}=(\mathbb{A}^{2})_{21\ldots 1}=0$, which implies $\mathbb{A}^{2}$ is not essential positive.

\noindent {\bf Case 2: }
$M(\mathbb{A})=\left(\begin{array}{cc}
0 & 1\\
1 & 0
\end{array}\right)$.

Then we have $M(\mathbb{A})_{11}=M(\mathbb{A})_{22}=0$, that is $a_{11\ldots 1}=a_{22\ldots 2}=0$,
by (\ref{eq35}) we have $M(\mathbb{A}^{2})_{12}=(\mathbb{A}^{2})_{12\ldots 2}=0$, which implies $\mathbb{A}^{2}$ is not essential positive.

Based on the above two cases and Proposition \ref{pro16}, we complete the proof of the necessity.
\end{proof}

A nature question is whether the result of Theorem \ref{thm32} is true for $n\geq 3$ or not.
The following Example \ref{exam33} shows that the necessity of Theorem \ref{thm32} is false with $n\geq 3$.

\begin{exam}\label{exam33}
Let $\mathbb{A}=(a_{i_{1}i_{2}\ldots i_{m}})$ be a nonnegative tensor of order $m$ and dimension $n\geq 3$, where
\begin{equation*}
a_{i_{1}i_{2}\ldots i_{m}}=\left\{\begin{array}{ll}
 0,   & \mbox{ if } i_{1}=1, i_{2}=i_{3}=\ldots =i_{m}\neq 1; \\
 1,   & \mbox{ otherwise}.
            \end{array}\right.
\end{equation*}
Then $\mathbb{A}$ is  (strongly) primitive, but $M(\mathbb{A})$ is not primitive.
\end{exam}
\begin{proof}
By direct calculation and Definition \ref{defn12}, we know that $\mathbb{A}^{2}$ is the tensor of order $(m-1)^{2}+1$ and dimension $n$, and for any $1\leq i\leq n$, we have
\begin{eqnarray*}
(\mathbb{A}^{2})_{i\alpha_{2}\alpha_{3}\ldots \alpha_{m}}&=& \sum\limits_{i_{2},i_{3},\ldots, i_{m}=1}^{2} a_{ii_{2}i_{3}\ldots i_{m}}a_{i_{2}\alpha_{2}}\ldots a_{i_{m}\alpha_{m}}\\
&\geq& \left\{\begin{array}{ll}
 a_{12\ldots m}a_{2\alpha_{2}}\ldots a_{m\alpha_{m}}=1,   &\mbox{ if }  i=1; \\
 a_{ii\ldots i}a_{i\alpha_{2}}\ldots a_{i\alpha_{m}}=1,  & \mbox{ otherwise}.
            \end{array}\right.
\end{eqnarray*}
Obviously, $\mathbb{A}^{2}$ is positive, then $\mathbb{A}$ is strongly primitive with $\eta(\mathbb{A})=2$
and thus $\mathbb{A}$ is primitive with $\gamma(\mathbb{A})=2$.

On the other hand, by the definition of $\mathbb{A}$, we have
$M(\mathbb{A})=\left(\begin{array}{cccc}
1 & 0 & \ldots & 0\\
1 & 1 & \ldots & 1\\
 &  & \vdots &  \\
1 & 1 & \ldots & 1
\end{array}\right)$. Since the associated digraph of $M(\mathbb{A})$ is not strongly connected, thus $M(\mathbb{A})$ is not primitive.
\end{proof}

Next, we will study the strongly primitive degree of order $m$ and dimension 2 tensors. Firstly, we  discuss an example with order $m=5$ and dimension $n=2$ tensor as follows.

\begin{defn} {\rm (See \cite{2013S2})} \label{defn20} Let $\mathbb{A}$ be a tensor with order $m$ and dimension $n$.
The $i$-th slice of $\mathbb{A}$,  denoted by  $\mathbb{A}[i]$, is the subtensor of $\mathbb{A}$ with  order $m-1$ and dimension $n$ such that $(\mathbb{A}[i])_{i_{2}\ldots i_{m}}=a_{ii_{2}\ldots i_{m}}$.
\end{defn}

\begin{exam}\label{exam35}
Let $\mathbb{A}=(a_{i_{1}i_{2}i_{3}i_{4}i_{5}})_{1\leq i_{j}\leq 2(j=1,\ldots, 5)}$
 be a nonnegative tensor with order $m=5$ and dimension $n=2$,
where $a_{12122}=a_{21121}=0$ and other $a_{i_{1}i_{2}i_{3}i_{4}i_{5}}=1$.
Then there exists at least one zero element in each slice of $\mathbb{A}^{2}$.
\end{exam}
\begin{proof}
Let $\alpha_{1}=2122$,  $\alpha_{2}=1121$, and denote $\beta_{2}= \beta_{4}=\beta_{5}=\alpha_{1}$,  $\beta_{3}=\alpha_{2}$.
Then we have

\hskip0.6cm $(A^{2})_{1\beta_{2}\beta_{3}\beta_{4}\beta_{5}}$

\noindent $\xlongequal{\quad\quad}\sum\limits_{i_{2},i_{3},i_{4},i_{5}=1}^{2} a_{1i_{2}i_{3}i_{4} i_{5}}a_{i_{2}\beta_{2}}a_{i_{3}\beta_{3}}a_{i_{4}\beta_{4}}a_{i_{5}\beta_{5}}$

 \noindent $\xlongequal {\quad\quad}\sum\limits_{i_{3},i_{4},i_{5}=1}^{2} a_{11i_{3}i_{4}i_{5}}a_{1\beta_{2}}a_{i_{3}\beta_{3}}a_{i_{4}\beta_{4}}a_{i_{5}\beta_{5}}
+\sum\limits_{i_{3},i_{4},i_{5}=1}^{2} a_{12i_{3}i_{4}i_{5}}a_{2\beta_{2}}a_{i_{3}\beta_{3}}a_{i_{4}\beta_{4}}a_{i_{5}\beta_{5}}$

\noindent $\xlongequal{a_{1\beta_{2}}=0}\sum\limits_{i_{3},i_{4},i_{5}=1}^{2} a_{12i_{3}i_{4}i_{5}}a_{2\beta_{2}}a_{i_{3}\beta_{3}}a_{i_{4}\beta_{4}}a_{i_{5}\beta_{5}}$

\noindent $\xlongequal{\quad\quad}\sum\limits_{i_{4} ,i_{5}=1}^{2} a_{121i_{4}i_{5}}a_{2\beta_{2}}a_{1\beta_{3}}a_{i_{4}\beta_{4}}a_{i_{5}\beta_{5}}+\sum\limits_{i_{4} ,i_{5}=1}^{2} a_{122i_{4}i_{5}}a_{2\beta_{2}}a_{2\beta_{3}}a_{i_{4}\beta_{4}}a_{i_{5}\beta_{5}}   $

\noindent $\xlongequal{a_{2\beta_{3}}=0}\sum\limits_{i_{4},i_{5}=1}^{2}a_{121i_{4}i_{5}}a_{2\beta_{2}}a_{1\beta_{3}}a_{i_{4}\beta_{4}}a_{i_{5}\beta_{5}}$

\noindent $\xlongequal{\quad\quad}\sum\limits_{i_{5}=1}^{2} a_{1211i_{5}}a_{2\beta_{2}}a_{1\beta_{3}}a_{1\beta_{4}}a_{i_{5}\beta_{5}}
+\sum\limits_{i_{5}=1}^{2} a_{1212i_{5}}a_{2\beta_{2}}a_{1\beta_{3}}a_{2\beta_{4}}a_{i_{5}\beta_{5}} $

\noindent $\xlongequal{a_{1\beta_{4}}=0}\sum\limits_{i_{5}=1}^{2}a_{1212i_{5}}a_{2\beta_{2}}a_{1\beta_{3}}a_{2\beta_{4}}a_{i_{5}\beta_{5}} $

\noindent $\xlongequal{\quad\quad}a_{12121}a_{2\beta_{2}}a_{1\beta_{3}}a_{2\beta_{4}}a_{1\beta_{5}}
+a_{12122}a_{2\beta_{2}}a_{1\beta_{3}}a_{2\beta_{4}}a_{2\beta_{5}} $

\noindent $ \xlongequal{a_{1\beta_{5}}=0} a_{12122}a_{2\beta_{2}}a_{1\beta_{3}}a_{2\beta_{4}}a_{2\beta_{5}}$

\noindent $\xlongequal{a_{12122}=0}0$.

Similarly, we let $\gamma_{2}=\gamma_{3}=\gamma_{5}=\alpha_{2}$ and $\gamma_{4}=\alpha_{1}$, 
we can show $(\mathbb{A}^{2})_{2\gamma_{2}\gamma_{3}\gamma_{4}\gamma_{5}}=0$ and we omit it.

Combining the above arguments, we know there exists at least one zero element in each slice of $\mathbb{A}^{2}$ by
$(A^{2})_{1\beta_{2}\beta_{3}\beta_{4}\beta_{5}}=(\mathbb{A}^{2})_{2\gamma_{2}\gamma_{3}\gamma_{4}\gamma_{5}}=0$.
\end{proof}

Similarly, the result of  Example \ref{exam35} can be generalized to any nonnegative tensor with order $m$ and dimension $n=2$.

\begin{lem}\label{lem36}
Let $\mathbb{A}=(a_{i_{1}i_{2}\ldots i_{m}})_{1\leq i_{j}\leq 2(j=1,\ldots, m)}$ be a nonnegative tensor with order $m$ and dimension $n=2$.
If there exist $\alpha_{1}=j_{2}j_{3}\ldots j_{m}\in [2]^{m-1}$ and $\alpha_{2}=k_{2}k_{3}\ldots k_{m}\in [2]^{m-1}$ such that
$a_{1\alpha_{1}}=a_{2\alpha_{2}}=0$. For any $2\leq t \leq m$, let $\beta_{t}=\left\{\begin{array}{cc}
                                                                                      \alpha_{2}, & \mbox{ if } j_{t}=1; \\
                                                                                       \alpha_{1}, & \mbox{ if } j_{t}=2,
                                                                                     \end{array}\right.$
and $\gamma_{t}=\left\{\begin{array}{cc}
                                                                                      \alpha_{2}, & \mbox{ if } k_{t}=1; \\
                                                                                       \alpha_{1}, & \mbox{ if } k_{t}=2.
                                                                                     \end{array}\right.$ Then

\begin{equation}\label{eq36}
(\mathbb{A}^{2})_{1\beta_{2}\ldots\beta_{m}}=0, \qquad  (\mathbb{A}^{2})_{2\gamma_{2}\ldots\gamma_{m}}=0.
\end{equation}
\end{lem}
\begin{proof}
We first show $(\mathbb{A}^{2})_{1\beta_{2}\ldots\beta_{m}}=0$. For any $2\leq t \leq m$, $j_{t}\in \{1,2\}$,
we denote $\overline{j_{t}}\in \{1,2\} \backslash \{j_{t} \}$.
Then we have $a_{\overline{j_{t}}\beta_{t}}=0$ 
by $a_{1\alpha_{1}}=a_{2\alpha_{2}}=0$, and

\hskip0.9cm$(\mathbb{A}^{2})_{1\beta_{2}\ldots\beta_{m}}$

\noindent$\xlongequal{\quad\quad\quad}\sum\limits_{i_{2},i_{3},\ldots, i_{m}=1}^{2} a_{1i_{2}i_{3}\ldots i_{m}}a_{i_{2}\beta_{2}}\ldots a_{i_{m}\beta_{m}} $

\noindent $\xlongequal{\quad\quad\quad}\sum\limits_{i_{3},\ldots, i_{m}=1}^{2} a_{11i_{3}\ldots i_{m}}a_{1\beta_{2}}a_{i_{3}\beta_{3}}\ldots a_{i_{m}\beta_{m}}
+\sum\limits_{i_{3},\ldots, i_{m}=1}^{2} a_{12i_{3}\ldots i_{m}}a_{2\beta_{2}}a_{i_{3}\beta_{3}}\ldots a_{i_{m}\beta_{m}}$

\noindent$\xlongequal{a_{\overline{j_{2}}\beta_{2}}=0}\sum\limits_{i_{3},\ldots, i_{m}=1}^{2} a_{1j_{2}i_{3}\ldots i_{m}}a_{j_{2}\beta_{j_{2}}}a_{i_{3}\beta_{3}}\ldots a_{i_{m}\beta_{m}}$

\noindent $\xlongequal{\quad\quad\quad}\sum\limits_{i_{4},\ldots, i_{m}=1}^{2} a_{1j_{2}1i_4\ldots i_{m}}a_{j_{2}\beta_{j_{2}}}a_{1\beta_{3}}a_{i_{4}\beta_{4}}\ldots a_{i_{m}\beta_{m}}$

\noindent \hskip1.5cm $+\sum\limits_{i_{4},\ldots, i_{m}=1}^{2} a_{1j_{2}2i_4\ldots i_{m}}a_{j_{2}\beta_{j_{2}}}a_{2\beta_{3}}a_{i_{4}\beta_{4}}\ldots a_{i_{m}\beta_{m}}$

\noindent $\xlongequal{{a_{\overline{j_{3}}\beta_{3}}=0}}\sum\limits_{i_{4},\ldots, i_{m}=1}^{2} a_{1j_{2}j_{3}i_{4}\ldots i_{m}}a_{j_{2}\beta_{j_{2}}}a_{j_{3}\beta_{j_{3}}}a_{i_{4}\beta_{4}}\ldots a_{i_{m}\beta_{m}}$

\noindent $\xlongequal{\quad\quad\quad}\ldots\ldots$

\noindent $\xlongequal{\quad\quad\quad}\sum\limits_{i_{m}=1}^{2} a_{1j_{2}j_{3}\ldots j_{m-1}i_{m}}a_{j_{2}\beta_{j_{2}}}a_{j_{3}\beta_{j_{3}}}\ldots a_{j_{m-1}\beta_{j_{m-1}}}a_{i_{m}\beta_{m}}$

\noindent $\xlongequal{\quad\quad\quad}a_{1j_{2}\ldots j_{m-1}1}a_{j_{2}\beta_{j_{2}}}\ldots a_{j_{m-1}\beta_{j_{m-1}}}a_{1\beta_{m}}+a_{1j_{2}\ldots j_{m-1}2}a_{j_{2}\beta_{j_{2}}}\ldots a_{j_{m-1}\beta_{j_{m-1}}}a_{2\beta_{m}}$

\noindent $\xlongequal{a_{\overline{j_{m}}\beta_{m}}=0}a_{1j_{2}\ldots j_{m}}a_{j_{2}\beta_{j_{2}}}\ldots a_{j_{m}\beta_{j_{m}}} $

\noindent $\xlongequal{\hskip.15cm a_{1\alpha_{1}}=0\hskip.15cm}0$.

Similarly, for any $2\leq t \leq m$, $k_{t}\in \{1,2\}$,  we denote $\overline{k_{t}}\in \{1,2\} \backslash \{k_{t} \}$.
Then we have $a_{\overline{k_{t}}\gamma_{t}}=0$ 
and we can show $(\mathbb{A}^{2})_{2\gamma_{2}\ldots\gamma_{m}}=0$ by
\begin{equation*}
(\mathbb{A}^{2})_{2\gamma_{2}\ldots\gamma_{m}}=\sum\limits_{i_{2},i_{3},\ldots, i_{m}=1}^{2}
a_{2i_{2}i_{3}\ldots i_{m}}a_{i_{2}\gamma_{2}}\ldots a_{i_{m}\gamma_{m}}
\end{equation*}
and the similar process of the above arguments. Thus we  complete the proof of (\ref{eq36}).
\end{proof}

\begin{them}\label{thm37}
Let $\mathbb{A}=(a_{i_{1}i_{2}\ldots i_{m}})_{1\leq i_{j}\leq 2(j=1,\ldots, m)}$ be a nonnegative tensor with order $m$ and dimension $n=2$.
If there exist $\alpha_{1}=j_{2}j_{3}\ldots j_{m}\in [2]^{m-1}$ and $\alpha_{2}=k_{2}k_{3}\ldots k_{m}\in [2]^{m-1}$ such that
$a_{1\alpha_{1}}=a_{2\alpha_{2}}=0$. Then $\mathbb{A}$ is not strongly primitive.
\end{them}
\begin{proof}
Now we show that there exists at least one zero element in each slice of $\mathbb{A}^{r}$ by induction on $r(\geq 2)$.

Firstly, by Lemma \ref{lem36}, we know there exists at least one zero element in each slice of $\mathbb{A}^{2}$.
Now we assume   that there exists at least one zero element in each slice of $\mathbb{A}^{r-1}$, say,
there exist $\delta_1, \delta_2\in [2]^{(m-1)^{r-1}}$ such that $(\mathbb{A}^{r-1})_{1\delta_1}=(\mathbb{A}^{r-1})_{2\delta_2}=0$.
For any $2\leq t \leq m$, let $\beta_{t}=\left\{\begin{array}{cc}
                                                                                      \delta_{2}, & \mbox{ if } j_{t}=1; \\
                                                                                       \delta_{1}, & \mbox{ if } j_{t}=2,
                                                                                     \end{array}\right.$
and $\gamma_{t}=\left\{\begin{array}{cc}
                                                                                      \delta_{2}, & \mbox{ if } k_{t}=1; \\
                                                                                       \delta_{1}, & \mbox{ if } k_{t}=2.
                                                                                     \end{array}\right.$
Then by (\ref{eq21}) and the similar proof  of Lemma \ref{lem36}, we have
  \begin{equation}\label{eq37}
(\mathbb{A}^{r})_{1\beta_{2}\ldots\beta_{m}}
=\sum\limits_{i_{2},i_{3},\ldots, i_{m}=1}^{n} a_{1i_{2}i_{3}\ldots i_{m}}
(\mathbb{A}^{r-1})_{i_{2}\beta_{2}}\ldots (\mathbb{A}^{r-1})_{i_{m}\beta_{m}}=0,
\end{equation}

\noindent and
\begin{equation}\label{eq38}
(\mathbb{A}^{r})_{2\gamma_{2}\ldots\gamma_{m}}
=\sum\limits_{i_{2},i_{3},\ldots, i_{m}=1}^{n} a_{2i_{2}i_{3}\ldots i_{m}}
(\mathbb{A}^{r-1})_{i_{2}\gamma_{2}}\ldots (\mathbb{A}^{r-1})_{i_{m}\gamma_{m}}=0.
\end{equation}

By (\ref{eq37}) and (\ref{eq38}), we  obtain there exists at least one zero element in each slice of $\mathbb{A}^{r}$,
and thus we  complete the proof.
\end{proof}

 Now we give the characterization of the strongly primitive tensor with order $m$ and dimension 2.

\begin{them}\label{thm38}
Let $\mathbb{A}=(a_{i_{1}i_{2}\ldots i_{m}})_{1\leq i_{j}\leq 2(j=1,\ldots, m)}$ be a nonnegative tensor with order $m$ and dimension $n=2$. Then

{\rm (1)} $\mathbb{A}$ is strongly primitive if and only if one of the following holds:

{\rm (a) } $\mathbb{A}=\mathbb{J}$;

{\rm (b) } $\mathbb{A}\not=\mathbb{J}$ and $a_{1i_{2}\ldots i_{m}}=a_{211\ldots 1}=1 (1\leq i_{j}\leq 2, j=2, \ldots,  m)$;

{\rm (c) }  $\mathbb{A}\not=\mathbb{J}$ and $a_{2i_{2}\ldots i_{m}}=a_{122\ldots 2}=1 (1\leq i_{j}\leq 2, j=2, \ldots,  m)$.

{\rm (2) }  If $\mathbb{A}$ is strongly primitive, then  $\eta(\mathbb{A})\leq 2$.
\end{them}
\begin{proof}
Firstly, we show    the sufficient of (1). It is easy to see that  $\mathbb{A}=\mathbb{J}$ is strongly primitive with $\eta(\mathbb{J})=1$,
and if $\mathbb{A}$ satisfies (b) or (c),
$\mathbb{A}$ is strongly primitive with $\eta(\mathbb{J})=2$ by Proposition \ref{prop27} immediately.

Now we show the necessity of (1), that is, if $\mathbb{A}$ is not satisfied the conditions of (a), (b) or (c),
then we will show that $\mathbb{A}$ is not strongly primitive. We complete the proof by the following three cases.

\noindent {\bf Case 1: } $a_{1\alpha}=1$ for any $\alpha \in[2]^{m-1}$ and $a_{211\ldots 1}=0$.

It is not difficult to find that $M(\mathbb{A})=\left(\begin{array}{cc}
1 & 1\\
0 & *
\end{array}\right)$. Then $\mathbb{A}$ is not primitive  by Theorem \ref{thm32}, and thus  $\mathbb{A}$ is not strongly primitive.

\noindent {\bf Case 2: } $a_{2\alpha}=1$ for any $\alpha \in[2]^{m-1}$ and $a_{122\ldots 2}=0$.

Similarly, we can find that $M(\mathbb{A})=\left(\begin{array}{cc}
* & 0\\
1 & 1
\end{array}\right)$. Then $\mathbb{A}$ is not primitive  by Theorem \ref{thm32}, and thus  $\mathbb{A}$ is not strongly primitive.

\noindent {\bf Case 3: } There  is at least one zero element in each slice of $\mathbb{A}$.

Then there exist  $\alpha_{1}=j_{2}j_{3}\ldots j_{m}\in [2]^{m-1}$ and $\alpha_{2}=k_{2}k_{3}\ldots k_{m}\in [2]^{m-1}$
such that $a_{1\alpha_{1}}=a_{2\alpha_{2}}=0$.
Thus $\mathbb{A}$ is not strongly primitive by Theorem \ref{thm37}.

(2)  If $\mathbb{A}$ is strongly primitive, by Definition \ref{defn111} and the proof of (1), we obtain $\eta(\mathbb{A})\leq 2$ immediately.
\end{proof}

\begin{rem}
 By Theorem \ref{thm38}, we can see   that the strongly primitive degree $\eta(\mathbb{A})$ of an nonnegative
  tensor with order $m$ and dimension $n=2$ is irrelevant to its order $m$.
  \end{rem}

\section{Some properties and problems of order $m$  dimension $n(\geq 3)$ strongly primitive tensors }
\hskip.6cm In this section,
we will study some properties of the  strongly primitive tensors with order $m$ and dimension $n\geq 3$ and propose some questions for further research.

\begin{prop}\label{prop41}
Let $\mathbb{A}=(a_{i_{1}i_{2}\ldots i_{m}})_{1\leq i_{j}\leq n(j=1,\ldots, m)}$ be a nonnegative tensor with order $m$ and dimension $n$.
Let  $s\in[n]$,   $2\leq t\leq m$,   $j^{(s)}_{{t}} \in [n]$,
and   $\alpha_{s}=j^{(s)}_{{2}}j^{(s)}_{{3}}\ldots j^{(s)}_{{m}}$.
 If $a_{i\alpha_{s}}=0$ for  any $i\in[n]$ and any $s\in [n]\backslash \{i\}$,
then there exist  $\gamma_1, \gamma_2, \ldots, \gamma_n\in [n]^{(m-1)^2}$ such that
$(\mathbb{A}^{2})_{i\gamma_k}=0$ for any $i\in [n]$ and any $k\in [n]\backslash \{i\}$.
\end{prop}
\begin{proof}
For each  $i \in [n]$, let  $\beta_t^{(k)}=\alpha_{j^{(k)}_t}\in [n]^{m-1}$ for any $k\in [n]$ and $2\leq t\leq m$,
then $a_{\overline{j^{(k)}_{t}}\beta^{(k)}_{t}}=0$ for any $\overline{j^{(k)}_{{t}}} \in [n]\backslash \{j^{(k)}_{{t}}\}$
by $a_{i\alpha_{s}}=0$ for  any $i\in[n]$ and any $s\in [n]\backslash \{i\}$.
Let $\gamma_k=\beta^{(k)}_{{2}}\beta^{(k)}_{{3}}\ldots\beta^{(k)}_{{m}}\in [n]^{(m-1)^2}$ for any $k\in [n]$.
Now we  show  $(\mathbb{A}^{2})_{i\gamma_k}=0$ for any $k\in [n]\backslash \{i\}$.

\hskip1.4cm $(\mathbb{A}^{2})_{i\gamma_k}$

\noindent $\xlongequal{\quad\quad\quad\quad} (\mathbb{A}^{2})_{i\beta^{(k)}_{2}\beta^{(k)}_{{3}}\ldots\beta^{(k)}_{{m}}}$

\noindent $\xlongequal{\quad\quad\quad\quad} \sum\limits_{i_{2},i_{3},\ldots, i_{m}=1}^{n} a_{ii_{2}i_{3}\ldots i_{m}}a_{i_{2}\beta^{(k)}_{2}}\ldots a_{i_{m}\beta^{(k)}_{m}} $

\noindent $\xlongequal{\quad\quad\quad\quad}  \sum\limits_{i_{3},\ldots, i_{m}=1}^{n} a_{i1i_{3}\ldots i_{m}}a_{1\beta^{(k)}_{2}}a_{i_{3}\beta^{(k)}_{3}}\ldots a_{i_{m}\beta^{(k)}_{m}}$

\noindent \hskip2.0cm $ +\ldots +\sum\limits_{i_{3},\ldots, i_{m}=1}^{n} a_{ini_{3}\ldots i_{m}}a_{n\beta^{(k)}_{2}}a_{i_{3}\beta^{(k)}_{3}}\ldots a_{i_{m}\beta^{(k)}_{m}}$

\noindent $\xlongequal{a_{\overline{j^{(k)}_{2}}\beta^{(k)}_{2}}=0} \sum\limits_{i_{3},\ldots, i_{m}=1}^{n} a_{ij^{(k)}_{2}i_{3}\ldots i_{m}}a_{j^{(k)}_{2}\beta^{(k)}_{2}} a_{i_{3}\beta^{(k)}_{3}}\ldots a_{i_{m}\beta^{(k)}_{m}}$

\noindent $\xlongequal{\quad\quad\quad\quad} \sum\limits_{i_{4},\ldots, i_{m}=1}^{n} a_{ij^{(k)}_{2}1i_4\ldots i_{m}}a_{j^{(k)}_{2}\beta^{(k)}_{2}}a_{1\beta^{(k)}_{3}}a_{i_{4}\beta^{(k)}_{4}}\ldots a_{i_{m}\beta^{(k)}_{m}} $

\noindent \quad\quad\quad\quad\quad $+ \ldots +\sum\limits_{i_{4},\ldots, i_{m}=1}^{n} a_{ij^{(k)}_{2}ni_4\ldots i_{m}}a_{j^{(k)}_{2}\beta^{(k)}_{2}}a_{n\beta^{(k)}_{3}}a_{i_{4}\beta^{(k)}_{4}}\ldots a_{i_{m}\beta^{(k)}_{m}}$

\noindent $\xlongequal{a_{\overline{j^{(k)}_{3}}\beta^{(k)}_{3}}=0}  \sum\limits_{i_{4},\ldots,i_{m}=1}^{n}  a_{ij^{(k)}_{2}j^{(k)}_{3}i_{4}\ldots i_{m}}a_{j^{(k)}_{2}\beta^{(k)}_{2}}a_{j^{(k)}_{3}\beta^{(k)}_{3}} a_{i_{4}\beta^{(k)}_{4}}\ldots a_{i_{m}\beta^{(k)}_{m}}$

\noindent $\xlongequal{\quad\quad\quad\quad}  \ldots$

\noindent $\xlongequal{\quad\quad\quad\quad} \sum\limits_{i_{m}=1}^{n}  a_{ij^{(k)}_{2}j^{(k)}_{3}\ldots j^{(k)}_{m-1}i_{m}}a_{j^{(k)}_{2}\beta^{(k)}_{2}}a_{j^{(k)}_{3}\beta^{(k)}_{3}}\ldots a_{j^{(k)}_{m-1}\beta^{(k)}_{m-1}}a_{i_{m}\beta^{(k)}_{m}}$

\noindent $\xlongequal{\quad\quad\quad\quad} a_{ij^{(k)}_{2}j^{(k)}_{3}\ldots j^{(k)}_{m-1}1}a_{j^{(k)}_{2}\beta^{(k)}_{2}}a_{j^{(k)}_{3}\beta^{(k)}_{3}}\ldots a_{j^{(k)}_{m-1}\beta^{(k)}_{m-1}}a_{1\beta^{(k)}_{m}} $

\noindent \quad\quad\quad\quad\quad $+ \ldots + a_{ij^{(k)}_{2}j^{(k)}_{3}\ldots j^{(k)}_{m-1}n}a_{j^{(k)}_{2}\beta^{(k)}_{2}}a_{j^{(k)}_{3}\beta^{(k)}_{3}}\ldots a_{j^{(k)}_{m-1}\beta^{(k)}_{m-1}}a_{n\beta^{(k)}_{m}}$

\noindent $\xlongequal{a_{\overline{j^{(k)}_{m}}\beta^{(k)}_{m}}=0}   a_{ij^{(k)}_{2}j^{(k)}_{3}\ldots j^{(k)}_{m-1}j^{(k)}_{m}}a_{j^{(k)}_{2}\beta^{(k)}_{2}}a_{j^{(k)}_{3}\beta^{(k)}_{3}}\ldots a_{j^{(k)}_{m-1}\beta^{(k)}_{m-1}}a_{j^{(k)}_{m}\beta^{(k)}_{m}}$

\noindent $\xlongequal{\quad\quad\quad\quad} a_{i\alpha_{k}}a_{j^{(k)}_{2}\beta^{(k)}_{2}}a_{j^{(k)}_{3}\beta^{(k)}_{3}}\ldots a_{j^{(k)}_{m-1}\beta^{(k)}_{m-1}}a_{j^{(k)}_{m}\beta^{(k)}_{m}}$

\noindent $\xlongequal{\quad a_{i\alpha_{k}}=0} 0.$

We note that $k\in [n] \backslash \{i \}$  which means there are $n-1$ zero elements in $i$-th slice of $\mathbb{A}^{2}$,
thus we complete the proof by $i \in [n]$.
\end{proof}

We note that Proposition \ref{prop41} is the generalization of Lemma \ref{lem36},
now we will obtain the  generalization of Theorem \ref{thm37}.

\begin{them}\label{thm42}
Let $\mathbb{A}=(a_{i_{1}i_{2}\ldots i_{m}})_{1\leq i_{j}\leq n(j=1,\ldots, m)}$ be a nonnegative tensor with order $m$ and dimension $n$.
Let  $ s\in[n]$,   $2\leq t\leq m$,   $j^{(s)}_{{t}} \in [n]$,
and   $\alpha_{s}=j^{(s)}_{{2}}j^{(s)}_{{3}}\ldots j^{(s)}_{{m}}$.
 If $a_{i\alpha_{s}}=0$ for  any $i\in[n]$ and any $s\in [n]\backslash \{i\}$,
then  $\mathbb{A}$ is not strongly primitive.
\end{them}

\begin{proof}
Now we show that there exist $\varepsilon_1, \varepsilon_2, \ldots, \varepsilon_n\in [n]^{(m-1)^{r}}$ such that
$(\mathbb{A}^{r})_{i\varepsilon_k}=0$ for any $i\in [n]$ and any  $k\in [n]\backslash \{i\}$ by induction on $r(\geq 2)$, say,
there exist at least  $n-1$ zero elements in each slice of $\mathbb{A}^{r}$ and thus $\mathbb{A}$ is not strongly primitive.

Firstly, by Lemma \ref{prop41}, we know  there exist  $\gamma_1, \gamma_2, \ldots, \gamma_n\in [n]^{(m-1)^2}$ such that
$(\mathbb{A}^{2})_{i\gamma_k}=0$ for any $i\in [n]$ and any $k\in [n]\backslash \{i\}$, say,
there exist at least  $n-1$ zero elements in each slice of $\mathbb{A}^{2}$.
Now we assume   that there exist $\delta_1, \delta_2, \ldots, \delta_n\in [n]^{(m-1)^{r-1}}$ such that $(\mathbb{A}^{r-1})_{i\delta_k}=0$
for any $i\in[n]$ and any $k\in [n]\backslash \{i\}$, say, there exist at least  $n-1$ zero elements in each slice of $\mathbb{A}^{r-1}$.

Let  $\eta_t^{(s)}=\delta_{j^{(s)}_t}$ for any $s\in [n]$ and $2\leq t\leq m$,
then $(\mathbb{A}^{r-1})_{\overline{j^{(s)}_{t}}\eta^{(s)}_{t}}=0$ for any $\overline{j^{(s)}_{{t}}} \in [n]\backslash \{j^{(s)}_{{t}}\}$
by $(\mathbb{A}^{r-1})_{i\delta_{k}}=0$ for  any $i\in[n]$ and any $k\in [n]\backslash \{i\}$.
Let $\varepsilon_k=\eta^{(k)}_{{2}}\eta^{(k)}_{{3}}\ldots\eta^{(k)}_{{m}}\in [n]^{(m-1)^{r}}$ for any $k\in [n]$.
Now we  show  $(\mathbb{A}^{r})_{i\varepsilon_k}=0$ for any $i\in [n]$ and any  $k\in [n]\backslash \{i\}$.

By (\ref{eq21}) and the similar proof  of Proposition \ref{prop41}, we have
  \begin{equation*}
(\mathbb{A}^{r})_{i\varepsilon_k}=(\mathbb{A}^{r})_{i\eta^{(k)}_{{2}}\eta^{(k)}_{{3}}\ldots\eta^{(k)}_{{m}}}
=\sum\limits_{i_{2},i_{3},\ldots, i_{m}=1}^{n} a_{ii_{2}i_{3}\ldots i_{m}}
(\mathbb{A}^{r-1})_{i_{2}\eta^{(k)}_{{2}}}\ldots (\mathbb{A}^{r-1})_{i_{m}\eta^{(k)}_{{m}}}=0,
\end{equation*}
then we  complete the proof.
\end{proof}

\begin{prop}\label{prop43}
Let $\mathbb{A}=(a_{i_{1}i_{2}\cdots i_{m}})_{1\leq i_{j}\leq n(j=1,\cdots ,m)}$ be a nonnegative tensor with order $m$ and dimension $n$, $M(\mathbb{A})$ be the majorization matrix of $\mathbb{A}$. If there exist $i,j \in [n]$, such that $(M(\mathbb{A}))_{ij} > 0, (M(\mathbb{A}))_{uj} = 0$ for any $u \in [n]\backslash \{i\}$, and $(M(\mathbb{A}))_{ji} > 0, (M(\mathbb{A}))_{vi} = 0$ for any $v \in [n]\backslash\{j\}$. Then $\mathbb{A}$ is not   primitive, and thus $\mathbb{A}$ is not   strongly primitive.
\end{prop}

\begin{proof}
Firstly, we show the following assert:\vskip.15cm

{\bf  If $k$ is odd, then $(M(\mathbb{A}^{k}))_{ij}>0,$ $(M(\mathbb{A}^{k}))_{ji}>0,$
 $(M(\mathbb{A}^{k}))_{uj}=0 \mbox{ for any } u \in [n]\backslash\{i\},$
$(M(\mathbb{A}^{k}))_{vi}=0  \mbox{ for any } v \in [n]\backslash \{j\}$.

 If $k$ is even, then $(M(\mathbb{A}^{k}))_{ii}>0,$ $(M(\mathbb{A}^{k}))_{jj}>0,$
 $(M(\mathbb{A}^{k}))_{ui}=0 \mbox{ for any } u \in [n]\backslash\{i\},$
$(M(\mathbb{A}^{k}))_{vj}=0  \mbox{ for any } v \in [n]\backslash \{j\}$.} \vskip.15cm

When $k=1$, the above result holds is obvious. When $k=2$,
by Definition \ref{defn17} and (\ref{eq21}), we have

 $(M(\mathbb{A}^{2}))_{ii} =(\mathbb{A}^{2})_{ii\ldots i}$

\hskip1.9cm $=\sum\limits_{i_{2},i_{3},\ldots ,i_{m}=1}^{n} a_{ii_{2}i_{3}\ldots i_{m}}a_{i_{2}i\ldots i}\ldots a_{i_{m}i\ldots i}$

\hskip1.9cm $= a_{ij\ldots j}(a_{ji\ldots i})^{m-1} $

\hskip1.9cm $=(M(\mathbb{A}))_{ij}((M(\mathbb{A}))_{ji})^{m-1}$

\hskip1.9cm $>0$,

\noindent and for any $u \in [n]\backslash \{i\}$, we have

  $(M(\mathbb{A}^{2}))_{ui} =(\mathbb{A}^{2})_{ui\ldots i}$

 \hskip1.9cm $=\sum\limits_{i_{2},i_{3},\ldots ,i_{m}=1}^{n} a_{ui_{2}i_{3}\ldots i_{m}}a_{i_{2}i\ldots i}\ldots a_{i_{m}i\ldots i}$

 \hskip1.9cm $=a_{uj\ldots j}(a_{ji\ldots i})^{m-1} $

\hskip1.9cm $=(M(\mathbb{A}))_{uj}((M(\mathbb{A}))_{ji})^{m-1}$

\hskip1.9cm $ = 0.$

Similarly,  we can show $(M(\mathbb{A}^{2}))_{jj} > 0$ and $(M(\mathbb{A}^{2}))_{vj}=0$ for any $v \in [n]\backslash \{j\}$.

Now we assume that for any $k$, the above assert holds. Then for $k+1$, we consider the following two cases.

\noindent {\bf Case 1: } $k$ is odd.

Then by (\ref{eq21}), we have

$(M(\mathbb{A}^{k+1}))_{ii}=(\mathbb{A}^{k+1})_{ii\ldots i}$

\hskip2.3cm $=\sum\limits_{i_{2},i_{3},\ldots ,i_{m}=1}^{n} a_{ii_{2}i_{3}\ldots i_{m}}(\mathbb{A}^{k})_{i_{2}i\ldots i}\ldots (\mathbb{A}^{k})_{i_{m}i\ldots i}$

\hskip2.3cm $=a_{ij\ldots j}((\mathbb{A}^{k})_{ji\ldots i})^{m-1}$

\hskip2.3cm $=(M(\mathbb{A}))_{ij}((M(\mathbb{A}^{k}))_{ji})^{m-1}$

\hskip2.3cm $>0,$

\noindent and for any $u \in [n]\backslash \{i\}$, we have

$(M(\mathbb{A}^{k+1}))_{ui}=(\mathbb{A}^{k+1})_{ui\ldots i}$

\hskip2.4cm $=\sum\limits_{i_{2},i_{3},\ldots ,i_{m}=1}^{n} a_{ui_{2}i_{3}\ldots i_{m}}(\mathbb{A}^{k})_{i_{2}i\ldots i}\ldots (\mathbb{A}^{k})_{i_{m}i\ldots i}$

\hskip2.4cm $=a_{uj\ldots j}((\mathbb{A}^{k})_{ji\ldots i})^{m-1}$

\hskip2.4cm $=(M(\mathbb{A}))_{uj}((M(\mathbb{A}^{k}))_{ji})^{m-1}$

\hskip2.4cm $=0.$

Similarly,  we can show $(M(\mathbb{A}^{k+1}))_{jj} > 0$ and $(M(\mathbb{A}^{k+1}))_{vj}=0$ for any $v \in [n]\backslash \{j\}$.

\noindent {\bf Case 2: } $k$ is even.

By (\ref{eq21}) and the similar proof of Case 1, we can show $ (M(\mathbb{A}^{k+1}))_{ij}>0$,
 $(M(\mathbb{A}^{k+1}))_{ji}>0,$
 $(M(\mathbb{A}^{k+1}))_{uj}=0 \mbox{ for any } u \in [n]\backslash\{i\},$
and $(M(\mathbb{A}^{k+1}))_{vi}=0  \mbox{ for any } v \in [n]\backslash \{j\}$.

By Proposition \ref{pro18} and the above assert, we know $\mathbb{A}$ is not   primitive, and thus $\mathbb{A}$ is not   strongly primitive.
\end{proof}

Let $\mathbb{A}=(a_{i_{1}i_{2}\cdots i_{m}})_{1\leq i_{j}\leq n(j=1,\cdots ,m)}$ be a nonnegative  strongly primitive tensor with order $m$ and dimension $n$.
When $n=2$,   we know $\eta(\mathbb{A})\leq 2$ by Theorem \ref{thm38}.
When $n\geq 3$, we donot know the value  or bound of $\eta(\mathbb{A})$.
Even $n=3$, we donot find out all strongly primitive tensors.
Thus we think it is not easy to obtain the value  or bound of $\eta(\mathbb{A})$.
Based on the  computation of   the case   $n=3$,  we propose the following problem for further research.

\begin{prob}
Let $n\geq 3$, $\mathbb{A}=(a_{i_{1}i_{2}\cdots i_{m}})_{1\leq i_{j}\leq n(j=1,\cdots ,m)}$ be a nonnegative  strongly primitive tensor with order $m$ and dimension $n$. Then $\eta(\mathbb{A})< (n-1)^2+1. $
\end{prob}

In \cite{2015H, 2014Y}, the authors gave some  algebraic characterizations of a nonnegative  primitive tensor,
and in  \cite{2015C}, the authors showed that a nonnegative tensor is primitive if and only if the greatest common divisor of all the
cycles in the associated directed hypergraph is equal to 1. It is natural for us to consider the following.
\begin{prob}
Study the algebraic or graphic characterization of a nonnegative strongly primitive tensor.
 \end{prob}

 We are sure the above two questions are interesting and not easy.


\end{document}